\documentclass[11pt]{article}

\usepackage{amsmath,amssymb,amsfonts,amsthm}

\usepackage{tikz}
\usepackage{graphicx}
\usepackage{placeins}
\usetikzlibrary{automata}
\usetikzlibrary{arrows}
\usetikzlibrary{positioning,calc}
\usetikzlibrary{graphs}
\usetikzlibrary{graphs.standard}
\usetikzlibrary{arrows,decorations.markings}
\usepackage{tkz-graph}
\usetikzlibrary{chains,fit,shapes}
\usetikzlibrary{calc}
\tikzset{every loop/.style={min distance=10mm,looseness=10}}
\tikzset{every state/.style={minimum size=2mm}}

\newtheorem{theorem}{Theorem}

\newtheorem{lemma}[theorem]{Lemma}

\newtheorem{conjecture}{Conjecture}
\newtheorem{Fact}[theorem]{Observation}

\title{On the representation number of grid graphs and cylindric grid graphs}

\author{Nawaf Shafi Alshammari\footnote{Department of Mathematics and Statistics, University of Strathclyde, 26 Richmond Street, Glasgow G1, 1XH, United Kingdom. 
{\bf Email:} nawaf.alshammari@strath.ac.uk.},\ 
Sergey Kitaev\footnote{Department of Mathematics and Statistics, University of Strathclyde, 26 Richmond Street, Glasgow G1, 1XH, United Kingdom. 
{\bf Email:} sergey.kitaev@strath.ac.uk.}\ \ and Artem Pyatkin\footnote{Sobolev Institute of Mathematics, Koptyug ave, 4, Novosibirsk, 630090, Russia\  {\bf Email:} artem@math.nsc.ru.}}

\begin{document}
	\maketitle

\begin{abstract}
The representation number of a graph is the minimum number of copies of each vertex required to represent the graph as a word, such that the letters corresponding to vertices $x$ and $y$ alternate if and only if $xy$ is an edge in the graph. It is known that path graphs, circle graphs, and ladder graphs have representation number 2, while prism graphs have representation number 3.

In this paper, we extend these results by showing that generalizations of the aforementioned graphs---namely, the $m \times n$ grid graphs and $m \times n$ cylindrical grid graphs---have representation number $3$ for $m \geq 3$ and $m \geq 2$, respectively, and $n \geq 3$. Furthermore, we discuss toroidal grid graphs in the context of word-representability, which leads to an interesting conjecture.\\[-3mm]

		\end{abstract}	

\section{Introduction}
An $m\times n$ {\em grid graph}  $\mbox{Gr}_{m,n}$ has vertex set $V=\{x_{ij}\ |\ i=1,\ldots,m, j=1,\ldots,n\}$ and edge set 
$E=\{x_{ij}x_{i,j+1}\ |\  i=1,\ldots,m, j=1,\ldots,n-1\} \cup \{x_{ij}x_{i+1,j}\ |\  i=1,\ldots,m-1, j=1,\ldots,n\}$. A {\em cylindric grid graph} $\mbox{CGr}_{m,n}$ is obtained from $\mbox{Gr}_{m,n}$ by adding edges $x_{i,1}x_{i,n}$ for all $i=1,\ldots,m$, where we assume that $n\geq 3$. For instance, the grid graph $\mbox{Gr}_{3,5}$ and the cylindric grid graph $\mbox{CGr}_{3,5}$ are shown in Figure~\ref{3x5-grid-graph}. In this paper, w.l.o.g., we assume that $m\leq n$. 

Note that $\mbox{Gr}_{1,n}=P_n$ is the path graph on $n$ vertices,  and $\mbox{Gr}_{2,n}$  is called a {\em ladder graph}. Hence, $\mbox{Gr}_{m,n}$ is a generalization of path graphs and ladder graphs. On the other hand, the cylindric grid graph  $\mbox{CGr}_{m,n}$ is a generalization of the cycle graph $C_n$ (the case of $m=1$) and the prism $\mbox{Pr}_n$ (the case of $m=2$).  Note also that the graphs $\mathrm{Gr}_{m,n}$ and $\mathrm{CGr}_{m,n}$ are the Cartesian products of the path graph $P_m$ with $P_n$ and $C_n$, respectively.

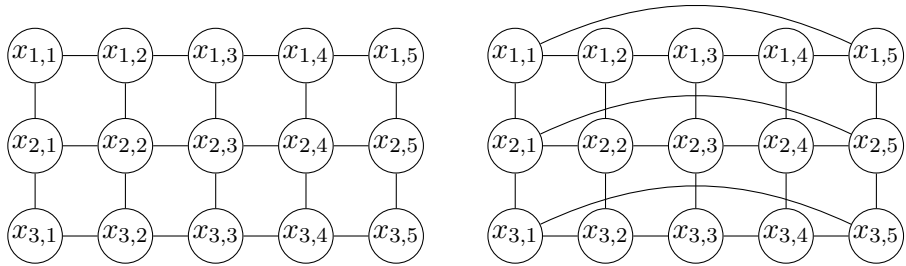
\begin{figure}
\begin{center}
\begin{tabular}{ccc}
\begin{tikzpicture}[scale=1.2, every node/.style={circle, draw, inner sep=0.5pt}]
  \def\rows{3}
  \def\cols{5}

  \foreach \i in {1,2,3} {
    \foreach \j in {1,2,3,4,5} {
      \node (x\i\j) at (\j,-\i) {$x_{\i,\j}$};
    }
  }

  \foreach \i in {1,2,3} {
    \foreach \j in {1,2,3,4} {
      \draw (x\i\j) -- (x\i\the\numexpr\j+1\relax);
    }
  }

  \foreach \i in {1,2} {
    \foreach \j in {1,2,3,4,5} {
      \draw (x\i\j) -- (x\the\numexpr\i+1\relax\j);
    }
  }
\end{tikzpicture}
& &
\begin{tikzpicture}[scale=1.2, every node/.style={circle, draw, inner sep=0.5pt}]
  \def\rows{3}
  \def\cols{5}

  \foreach \i in {1,2,3} {
    \foreach \j in {1,2,3,4,5} {
      \node (x\i\j) at (\j,-\i) {$x_{\i,\j}$};
    }
  }

  \foreach \i in {1,2,3} {
    \foreach \j in {1,2,3,4} {
      \draw (x\i\j) -- (x\i\the\numexpr\j+1\relax);
    }
  }

  \foreach \i in {1,2} {
    \foreach \j in {1,2,3,4,5} {
      \draw (x\i\j) -- (x\the\numexpr\i+1\relax\j);
    }
  }
  
\draw[bend left=25] (x11) to (x15); 
\draw[bend left=25] (x21) to (x25); 
\draw[bend left=25] (x31) to (x35); 
  
\end{tikzpicture}

\end{tabular}
\caption{The grid graph $\mbox{Gr}_{3,5}$ and the cylindric grid graph $\mbox{CGr}_{3,5}$}\label{3x5-grid-graph}
\end{center}
\end{figure}

A graph $G=(V,E)$ is {\em word-representable} if there exists a word $w$ over the alphabet $V$ such that two distinct letters $x$ and $y$ alternate in $w$ if and only if $xy\in E$. It is known \cite{KL15,KP18} that any word-representable graph $G$  is {\em $k$-word-representable} for some $k$; that is, there exists  a word $w$ representing $G$ in which each letter appears exactly $k$ times. The minimum such $k$ is called the {\em representation number} of $G$, and it is denoted by $\mathcal{R}(G)$. For a non-word-representable graph $G$, by definition, $\mathcal{R}(G)=\infty$.

The representation number of graphs has been studied in the literature for various classes of graphs -- for example,  {\em crown graphs} \cite{GKP18,HHMO24} and the {\em $k$-dimensional cube} \cite{Bro18,BroZan19,HHMO24}. It is known \cite{KL15,KP18} that the class of {\em circle graphs}, excluding complete graphs, coincides precisely with the class of graphs having representation number 2. In particular, $\mathcal{R}(P_n)=\mathcal{R}(\mbox{Gr}_{1,n})=2$ and $\mathcal{R}(C_n)=\mathcal{R}(\mbox{CGr}_{1,n})=2$ for $n\geq 3$. Of importance to our work are studies on graphs with representation number 3, as conducted in \cite{Kit13,KitPya08} (see also Section~5.2 in \cite{KL15} for a summary). In particular,  the {\em Petersen graph} and any prism $\mbox{Pr}_n$ have representation number 3. The latter fact (presented in \cite{KitPya08}) shows that $\mathcal{R}(\mbox{CGr}_{m,n})\geq 3$ for $m\geq 2$ (and $n\geq 3$) since then $\mbox{Pr}_n$ is an induced subgraph of $\mbox{CGr}_{m,n}$. But is it true that  $\mathcal{R}(\mbox{CGr}_{m,n})=3$ for any $m\geq 2$ and $n\geq 3$?

Moreover, it is known that the representation number of a ladder graph is 2 \cite{Kit13},  so $\mathcal{R}(\mbox{Gr}_{2,n})=2$ (for $n\geq 2$) and  $\mathcal{R}(\mbox{Gr}_{2,1})=\mathcal{R}(\mbox{Gr}_{1,1})=1$ since then we deal with a complete graph. But what can be said about the representation number of a general grid graph $\mbox{Gr}_{m,n}$? 

The main results of this paper are the following two theorems, which extend the known results for the cases $m = 1, 2$.

\begin{theorem}\label{main1}  For $m,n\ge 3$, $\mathcal{R}(\mbox{Gr}_{m,n})=3$.\end{theorem}

\begin{theorem}\label{main2}  For $m,n\ge 3$, $\mathcal{R}(\mbox{CGr}_{m,n})=3$.\end{theorem}

The paper is organized as follows. In Sections~\ref{Gr-sec} and~\ref{CGr-sec}, we provide the proofs of Theorem~\ref{main1} (which follows from Lemmas~\ref{lem-Gr33} and~\ref{main}) and Theorem~\ref{main2} (which follows from Theorems~\ref{cyl3} and~\ref{cyl4}, as well as the non-2-representability of prisms mentioned above), respectively. In Section~\ref{open}, we briefly discuss the word-representability of toroidal grid graphs and state a conjecture on this topic.

\section{3-representability of grid graphs}\label{Gr-sec}

We begin with proving the lower bound on the representation number of the grid graphs, that follows from the fact that  $\mbox{Gr}_{3,3}$ is an induced subgraph of any larger grid graph and the next lemma.

\begin{lemma}\label{lem-Gr33} We have $\mathcal{R}(\mbox{Gr}_{3,3})\geq 3$. \end{lemma} 

\begin{proof} Clearly, $\mathcal{R}(\mbox{Gr}_{3,3})\neq 1$, so it is sufficient to prove that $\mathcal{R}(\mbox{Gr}_{3,3})\neq 2$. 

The vertex $x_{2,2}$ is adjacent with vertices $x_{1,2},x_{2,1},x_{2,3}, x_{3,2}$, which are pairwise non-adjacent. If $\mbox{Gr}_{3,3}$ is 2-representable, then by symmetry, the closed neighbourhood of the vertex $x_{2,2}$  is represented by one of the following three words:
$$x_{2,2}x_{2,1}x_{1,2}x_{2,3}x_{3,2}x_{2,2}x_{3,2}x_{2,3}x_{1,2}x_{2,1},$$
$$x_{2,2}x_{2,1}x_{1,2}x_{3,2}x_{2,3}x_{2,2}x_{2,3}x_{3,2}x_{1,2}x_{2,1},$$
$$ x_{2,2}x_{2,1}x_{2,3}x_{1,2}x_{3,2}x_{2,2}x_{3,2}x_{1,2}x_{2,3}x_{2,1}.$$
In either way, one of the two copies of $x_{3,1}$ must be between two $x_{3,2}$, while the other copy of  $x_{3,1}$  must be outside of two  $x_{2,1}$. But then the vertex  $x_{3,1}$  is adjacent with  $x_{1,2}$, 
which is a contradiction. Hence,  $\mathcal{R}(\mbox{Gr}_{3,3})\geq 3$.
\end{proof}

If a letter $x$ occurs several times in a word $w$, denote by $x^j$ the $j$th occurence of $x$ in $w$.  In what follows, a {\it factor} of a word is a contiguous subsequence of the word. For example, all factors of the word $w=abca$ are $a^1$, $b^1$, $c^1$, $a^2$, $a^1b^1$, $b^1c^1$, $c^1a^2$, $a^1b^1c^1$, $b^1c^1a^2$, and $a^1b^1c^1a^2$.
We write \( x^i < y^j \) if the \( i \)th occurrence of the letter \( x \) appears to the left of the \( j \)th occurrence of the letter \( y \) in the word. In this case, we denote by \( [x^i, y^j] \) the factor induced by all letters lying between \( x^i \) and \( y^j \).

The next observation follows directly from the definitions.

\begin{Fact}~\label{evident}
If $xy\in E$ then $x^i>y^{i-1}$ and $y^i>x^{i-1}$ for all $i$.
\end{Fact}

Next, we prove that every grid graph is 3-representable. Let
$$W=x_1^1x_2^1x_1^2x_3^1x_2^2x_1^3x_4^1x_3^2x_2^3x_5^1x_4^2x_3^3\ldots x_{n}^1x_{n-1}^2x_{n-2}^3x_{n}^2x_{n-1}^3x_{n}^3.$$

Clearly, \( W \) is 3-uniform and represents the path \( P_n \) having the vertices in $\{x_1,x_2,\ldots,x_n\}$. Let $k=\lfloor n/2 \rfloor$. The following properties of the word $W$ are easy to check:

(a) For every $t = 2, \ldots, k$ the inequality $x_{2t}^2 > x_{2t-2}^3$ holds;

(b) For every $t = 1, \ldots, n-1$ and $i=1,2,3$ the inequality $x_t^i < x_{t+1}^i$ holds.



\begin{lemma}\label{main}  The grid graph $\mathrm{Gr}_{m,n}$ is $3$-representable for every $n \ge 3$. Moreover, there exists a word-representant such that for every $j=1,\ldots,m$

{\rm (1)} the subword $W_j$ induced by $x_{j,1},\ldots,x_{j,n}$ coincides with $W$ after replacing $x_{j,i}$ by $x_i$ for all $i$;

{\rm (2)} there are factors  $x_{j,2t+1}^2 x_{j,2t}^3$ for all $t = 1, \ldots, k - 1$ as well as the factor $x_{j,n}^2x_{j,n-1}^3$.
\end{lemma}



\begin{proof}
We apply a recursive construction based on the value of $m$. At each step we obtain the word $w_i$ representing $\mathrm{Gr}_{i,2k}$ by merging the words $w_{i-1}$ (representing $\mathrm{Gr}_{i-1,2k}$) and $W$. In particular, for $m = 1$, we let
$$w_1=x_{1,1}^1x_{1,2}^1x_{1,1}^2x_{1,3}^1x_{1,2}^2x_{1,1}^3x_{1,4}^1x_{1,3}^2x_{1,2}^3\ldots x_{1,n}^1x_{1,n-1}^2x_{1,n-2}^3x_{1,n}^2x_{1,n-1}^3x_{1,n}^3.$$

Clearly, $w_1$  coincides with $W$ (after the replacement of letters) and satisfies property~(2) of Lemma~\ref{main}.

Having a 3-uniform word $w_{m-1}$ representing $\mbox{Gr}_{m-1,2k}$, we obtain the word $w_m$ by the following rules. 

We substitute $x_{m-1,1}^2$  by $x_{m,1}^1x_{m-1,1}^2x_{m,2}^1x_{m,1}^2$; for all $t=1,\ldots,k-1$ we substitute the factor 
$ x_{m-1,2t+1}^2x_{m-1,2t}^3$ by
$$x_{m,2t+1}^1x_{m-1,2t+1}^2x_{m,2t}^2x_{m,2t-1}^3x_{m,2t+2}^1x_{m-1,2t}^3x_{m,2t+1}^2x_{m,2t}^3. \eqno(1)$$
If \( n \) is odd, we additionally substitute the factor \( x_{m-1,n}^2 x_{m-1,n-1}^3 \) with
\[
x_{m,n}^1 x_{m-1,n}^2 x_{m,n-1}^2 x_{m-1,n-1}^3 x_{m,n-2}^3. \eqno{(2)}
\] Finally, we substitute
$x_{m-1,n}^3$  by $x_{m,n}^2x_{m,n-1}^3x_{m-1,n}^3x_{m,n}^3$.

It is straightforward to verify that $w_m\setminus w_{m-1}$ coincides with $W$ (after the replacement of letters) and property~(2) of Lemma~\ref{main} is true for $w_m$.


Let us show that $w_m$ represents $\mbox{Gr}_{m,n}$. 
Since $w_m\setminus w_{m-1}$ coincides with $W$, 
the vertices $x_{m,1},\ldots, x_{m,n}$ induce the path graph $P_n$.

By (1), for every $t=1,\ldots,k-1$ we have $[x_{m,2t}^2,x_{m,2t}^3]\cap w_{m-1}=\{x_{m-1,2t}^3\}$ and $[x_{m,2t+1}^1,x_{m,2t+1}^2]\cap w_{m-1}=\{x_{m-1,2t+1}^2,x_{m-1,2t}^3\}$. By  Observation~\ref{evident}, this means that $x_{m,2t}$ and $x_{m,2t+1}$ may be adjacent in $w_{m-1}$ only to $x_{m-1,2t}$ and $x_{m-1,2t+1}$, respectively. Due to 
$[x_{m,n}^2,x_{m,n}^3]\cap w_{m-1}=\{x_{m-1,n}^3\}$ and $[x_{m,1}^1,x_{m,1}^2]\cap w_{m-1}=\{x_{m-1,1}^2\}$, the same fact holds for the vertices $x_{m,n}$ and $x_{m,1}$. 
Finally, if $n$ is odd, then by~(2) $[x_{m,n-1}^2,x_{m,n-1}^3]\cap w_{m-1}=\{x_{m-1,n-1}^3\}$ and so, $x_{m,n-1}$ may be adjacent in $w_{m-1}$ only to $x_{m-1,n-1}$.
Let us prove that these edges indeed exist.

Clearly, for all  $t=1,\ldots,k$ we have $x_{m,2t}^3>x_{m-1,2t}^3>x_{m,2t}^2$.  By property~(a) of \( W \), we also have \( x_{m,2t}^1 < x_{m-1,2t-2}^3 < x_{m-1,2t}^2 \) for all \( t = 2, \ldots, k \). The inequality $x_{m,2}^1<x_{m-1,2}^2$ folows from the facts that $w_m$ has a factor $x_{m-1,1}^2x_{m,2}^1$ and $x_1^2<x_2^2$ in $W$.
By~(1) and property~(a) of $W$, $x_{m,2t}^2>x_{m-1,2t+1}^2>x_{m-1,2t}^2$ for all $t=1,\ldots,k-1$. The inequality $x_{m,2k}^2>x_{m-1,2k}^2$ follows from~(2) if $n$ is odd, or from the fact that $w_{m-1}$ ends with $x_{m-1,n}^3$ if $n$ is even.
For all $t=2,\ldots,k$ by (1) and Observation~\ref{evident}, $x_{m,2t}^1>x_{m-1,2t-1}^2>x_{m-1,2t}^1$. Finally, by  Observation~\ref{evident}, $x_{m,2}^1>x_{m-1,1}^2>x_{m-1,2}^1$. So, the edge $x_{m,2t}x_{m-1,2t}$ exists for each  $t=1,\ldots,k$.

Since $w_n$ has the factor  $x_{m,1}^1x_{m-1,1}^2x_{m,2}^1x_{m,1}^2$, we have $x_{m-1,1}^1<x_{m,1}^1<x_{m-1,1}^2<x_{m,1}^2<x_{m-1,1}^3$.
Similarly, by (1), $x_{m-1,2t+1}^1<x_{m,2t+1}^1<x_{m-1,2t+1}^2<x_{m,2t+1}^2<x_{m-1,2t+1}^3$ for all $t=1,\ldots,k-1$. Since $x_{m-1,2t+3}^2x_{m-1,2t+2}^3$ is a factor in $w_{m-1}$ by property~(2) and 
$x_{m-1,2t+1}^3<x_{m-1,2t+3}^3$, we have $x_{m-1,2t+3}^2>x_{m-1,2t+1}^3$. Hence, by (1), for all $t=0,\ldots,k-2$ we get $x_{m,2t+1}^3>x_{m-1,2t+3}^2>x_{m-1,2t+1}^3$.
If $n$ is odd, then by~(2) and property~(b), $x_{m,n-2}^3> x_{m-1,n-1}^3>x_{m-1,n-2}^3$. Otherwise, $x_{m,n-1}^3>x_{m-1,n-1}^3$ because $w_n$ has a factor 
$x_{m,n-1}^3x_{m-1,n}^3$ and property~(b) holds for $W$. So, $x_{m,2t+1}x_{m-1,2t+1} \in E$ for all $t=0,\ldots,k-1$. Finally, if $n$ is odd then the inequalities
$x_{m-1,n}^1<x_{m,n}^1<x_{m-1,n}^2<x_{m,n}^2<x_{m-1,n}^3<x_{m,n}^3$ are easy to check directly.
\end{proof}

\section{3-representability of cylindric grid graphs}\label{CGr-sec}
In this section, we prove that $\mathrm{CGr}_{m,n}$ is $3$-representable for all $m, n \geq 3$. We begin with the case $n = 3$.

\begin{theorem}\label{cyl3}  Cylindrical grid graph $\mbox{CGr}_{m,3}$ is $3$-representable for all $m\geq 1$.\end{theorem}

\begin{proof}
We apply a recursive construction based on the value of $m$. In case of $m=1$, we let
$$w_1=x_{1,1}^1x_{1,2}^1x_{1,3}^1x_{1,1}^2x_{1,2}^2x_{1,3}^2x_{1,1}^3x_{1,2}^3x_{1,3}^3.$$
Evidently, $w_1$ represents $C_3=\mbox{CGr}_{1,3}$.

Having a 3-uniform word $w_{m-1}$ representing $\mbox{CGr}_{m-1,3}$, we obtain the word $w_m$ by the following rules: substitute
\begin{eqnarray*}
x_{m-1,1}^2 & \mbox{ by } &x_{m,1}^1x_{m,2}^1x_{m-1,1}^2\\
x_{m-1,3}^2 & \mbox{ by } & x_{m,3}^1x_{m-1,3}^2x_{m,1}^2; \\
x_{m-1,2}^3 & \mbox{ by } &x_{m,2}^2x_{m,3}^2x_{m,1}^3x_{m-1,2}^3x_{m,2}^3;\\
x_{m-1,3}^3 & \mbox{ by } &x_{m-1,3}^3x_{m,3}^3.
\end{eqnarray*}
Clearly, 
$$w_m\setminus w_{m-1}=x_{m,1}^1x_{m,2}^1x_{m,3}^1x_{m,1}^2x_{m,2}^2x_{m,3}^2x_{m,1}^3x_{m,2}^3x_{m,3}^3$$
represents $C_3$. Note that $x_{m,1}^i<x_{m,2}^i<x_{m,3}^i$ for all $m$ and $i=1,2,3$.

It is straightforward to check that the edges $x_{m,1}x_{m-1,1}, x_{m,2}x_{m-1,2}$ and $x_{m,3}x_{m-1,3}$ exist.

Since $x_{m-2,3}^2<x_{m-1,1}^2$ and $w_m$ contains the factor $x_{m,1}^1x_{m,2}^1x_{m-1,1}^2$, we have $x_{m,i}^1>x_{k,j}^2$ for all $k<m-1$ and  $i,j=1,2,3$. Hence, by Observation~\ref{evident}, no $x_{m,i}^1$ may be adjacent to $x_{k,j}^2$ for $k<m-1$.

Note that $x_{m,1}^2 > x_{m-1,3}^2 > x_{m-1,2}^2$, but $x_{m,1}^3 < x_{m-1,2}^3 < x_{m-1,3}^3$. Therefore, $x_{m,1}$ is not adjacent to $x_{m-1,2}$ or $x_{m-1,3}$. Clearly, $[x_{m,2}^2, x_{m,2}^3] \cap w_{m-1} = \{x_{m-1,2}^3\}$. Due to the presence of the factor $x_{m,3}^1 x_{m-1,3}^2$ in $w_m$, we have $x_{m,3}^1 > x_{m-1,2}^2 > x_{m-1,1}^2$. By Observation~\ref{evident}, this means that $x_{m,3}$ is not adjacent to $x_{m-1,2}$ or $x_{m-1,1}$.
\end{proof}

Prior to proving the result for arbitrary $n \geq 4$, we introduce two words:
$$\mbox{\texttt{Od}}=x_1^1x_n^1x_2^1x_1^2x_3^1x_2^2\ldots x_{n-1}^1x_{n-2}^2x_n^2x_{n-1}^2x_1^3x_2^3\ldots x_{n-2}^3x_n^3x_{n-1}^3;$$
$$\mbox{\texttt{Ev}}=x_1^1x_2^1\ldots x_{n-2}^1x_n^1x_{n-1}^1x_1^2x_n^2x_2^2x_1^3x_3^2x_2^3\ldots x_{n-1}^2x_{n-2}^3x_n^3x_{n-1}^3.$$

Clearly, both these 3-uniform words represent the cycle $C_n$, and \texttt{Ev} is obtained from \texttt{Od} by shifting the permutation $x_1^3x_2^3\ldots x_{n-2}^3x_n^3x_{n-1}^3$
at the beginning of the word. The following properties are easy to check:

(1) Both words start with $x_1$ and satisfy $x_n^i<x_{n-1}^i$  for all $i=1,2,3$.

(2) In \texttt{Od}, $x_{n-2}^1<x_n^2$ while in \texttt{Ev}, $x_{j+1}^2<x_j^3$ for all $j=1,\ldots,n-1$.

(3) In \texttt{Od}, $x_i^3>x_j^2$ while in \texttt{Ev}, $x_i^2>x_j^1$ for all  $i,j=1,\ldots,n$.

\begin{theorem}\label{cyl4}  Cylindrical grid graph $\mbox{CGr}_{m,n}$ is $3$-representable for all $m\geq 1$ and $n\ge 4$.\end{theorem}

\begin{proof}
For each $m$ we construct a 3-uniform word $w_m$ reprsenting $\mbox{CGr}_{m,n}$ and let $W_n=w_n\setminus w_{n-1}$. In our construction $W_m$ coincides, after replacing $x_{m,i}$ with  $x_i$ for all $j$, with \texttt{Od} if $m$ is odd and with \texttt{Ev} if $m$ is even. Namely, we start with
$$w_1=x_{11}^1x_{1n}^1x_{12}^1x_{11}^2\ldots x_{1,n-1}^1x_{1,n-2}^2x_{1,n}^2x_{1,n-1}^2x_{11}^3\ldots x_{1,n-2}^3x_{1,n}^3x_{1,n-1}^3$$
and apply one of the following constructions for $w_m$ depending on the parity of $m\ge 2$. 

If $m$ is even, we substitute:
\begin{eqnarray*}
x_{m-1,j}^2 & \mbox{ by } & x_{m,j}^1x_{m-1,j}^2  \mbox{ for all } j=1,\ldots,n;\\
x_{m-1,1}^3 & \mbox{ by } & x_{m,1}^2x_{m,n}^2x_{m-1,1}^3; \\
x_{m-1,j}^3 & \mbox{ by } & x_{m,j}^2x_{m,j-1}^3x_{m-1,j}^3 \mbox{ for all } j=2,\ldots,n-2;\\
x_{m-1,n}^3 & \mbox{ by } & x_{m,n-1}^2x_{m,n-2}^3x_{m-1,n}^3;\\
x_{m-1,n-1}^3 & \mbox{ by } & x_{m,n}^3x_{m-1,n-1}^3x_{m,n-1}^3.
\end{eqnarray*}

If $m$ is odd, we substitute:
\begin{eqnarray*}
x_{m-1,1}^2 & \mbox{ by } & x_{m,1}^1x_{m-1,1}^2;\\
x_{m-1,n}^2 & \mbox{ by } & x_{m,n}^1x_{m-1,n}^2 ;\\
x_{m-1,j}^2 & \mbox{ by } & x_{m,j}^1x_{m,j-1}^2x_{m-1,j}^2 \mbox{ for all } j=2,\ldots,n-1;\\
x_{m-1,n}^3 & \mbox{ by } & x_{m,n}^2x_{m,n-1}^2x_{m,1}^3x_{m-1,n}^3;\\
x_{m-1,n-1}^3 & \mbox{ by } & x_{m,2}^3x_{m,3}^3\ldots x_{m,n-2}^3x_{m,n}^3x_{m-1,n-1}^3x_{m,n-1}^3.
\end{eqnarray*}

Let us show that $w_m$ indeed represents $\mathrm{CGr}_{m,n}$. Clearly, $W_m$ coincides with \texttt{Od} and \texttt{Ev} for odd and even values of $m$, respectively; hence, the vertices $x_{m,1}, \ldots, x_{m,n}$ induce the cycle $C_n$. We have to verify that $x_{m,j}$ is adjacent in $w_{m-1}$ to  $x_{m-1,j}$ only for all $ j=1,\ldots,n$. 
Throughout the proof we will use the following easy observation: if $x<y$ and there is a factor $zy$ then $x<z$.
Consider two cases. \\[-3mm]

\noindent
{\bf Case 1.} Let $m\ge 3$ be odd. Note that $w_m$ contains a factor  $ x_{m,1}^1x_{m-1,1}^2$ while $w_{m-1}$ contains a factor $x_{m-1,1}^2x_{m-1,n}^2x_{m-2,1}^3$ (remind that  $w_{m-1}$ was constructed by the rules of even $m-1$). Therefore, $[x_{m,1}^1,x_{m-2,1}^3]\cap w_{m-2}=\emptyset.$ By property~3), $x_{m-2,1}^3>x_{m-2,j}^2$, and thus,
 $x_{m,1}^1>x_{m-2,j}^2$ for all $j=1,\ldots,n$. Hence, by Observation~\ref{evident}, no $x_{m,j}$ is adjacent to any $x\in w_{m-2}$.

Clearly, $[x_{m,n}^2,x_{m,n}^3]\cap w_{m-1}=\{x_{m-1,n}^3\}$; so, $x_{m,n}$ maybe adjacent in $w_{m-1}$ only to $x_{m-1,n}$. We have $[x_{m,1}^1,x_{m,1}^2]\cap w_{m-1} =
\{ x_{m-1,1}^2,x_{m-1,n}^2 \}$, but $x_{m,1}^3<x_{m-1,n}^3$. So, $x_{m,1}$ and $x_{m-1,n}$ are not adjacent. Similarly, 
$[x_{m,n-1}^2,x_{m,n-1}^3]\cap w_{m-1} = \{ x_{m-1,n}^3,x_{m-1,n-1}^3 \}$, but $x_{m-1,n}^2<x_{m,n-1}^1$ and hence  $x_{m,n-1}$ and $x_{m-1,n}$ are not adjacent. 
Let $j\in [2,n-2]$. Note that in  \texttt{Ev} we have $[x_j^2,x_{j+1}^2] =\{ x_{j-1}^3\}$. Therefore, $[x_{m,j}^1,x_{m,j}^2]\cap w_{m-1} =\{ x_{m-1,j}^2,x_{m-1,j-1}^3\}$.
Since by Observation~\ref{evident}, $x_{m,j}$ cannot be adjacent in to $x_{m-1,j-1}$, the only possible neighbor of  $x_{m,j}$ in $w_{m-1}$ is $x_{m-1,j}$.

It remains to verify that the edges $x_{m,j}x_{m-1,j}$ exist for all $j=1,\ldots,n$. Clearly,  $x_{m-1,j}^3<x_{m,j}^3$ for all $j$. Since
$[x_{m,j}^1,x_{m-1,j}^2]\cap w_{m-1}$ is either empty or contains $x_{m,j-1}^2$ only, we have $x_{m-1,j}^1<x_{m,j}^1<x_{m-1,j}^2<x_{m,j}^2$. 
The inequalities $x_{m,n}^2<x_{m-1,n}^3$ and  $x_{m,n-1}^2<x_{m-1,n-1}^3$ are evident. Finally, for all $j=1,\ldots,n-2$ by property~(2) we have
$x_{m,j}^2<x_{m-1,j+1}^2<x_{m-1,j}^3$.  So, $x_{m,j}x_{m-1,j}\in E$  for all $j=1,\ldots, n$. \\[-3mm]

\noindent
{\bf Case 2. } Let $m\ge 2$ be even. If $m\ge 4$ then by the previous case, $x_{m,1}^1>x_{m-1,1}^1>x_{m-3,j}^2$ for all $j$ and so, there are no edges between 
 $x_{m,i}$ and $w_{m-3}$. 

Note that $w_m$ contains a factor  $x_{m,1}^2x_{m,n}^2x_{m-1,1}^3$ while $w_{m-1}$ contains a factor $x_{m-1,1}^3x_{m-2,n}^3$. Hence,
 for all $i=1,\ldots,n$ and $j=1,\ldots,n-2$ we have $x_{m,i}^2>x_{m,1}^2>x_{m-2,j}^3$ and by Observation~\ref{evident}, $x_{m,i}x_{m-2,j}\not\in E$. 
Since $w_m$ has a factor $x_{m,1}^1x_{m-1,1}^2$ and $w_{m-1}$ has factors  $x_{m-1,1}^2x_{m-2,2}^2$ and $x_{m-1,n}^1x_{m-2,n}^2$, and $x_n^1<x_2^2$ in \texttt{Od}, we derive
$x_{m,i}^1>x_{m,1}^1>x_{m-2,n}^2$ and  by Observation~\ref{evident}, $x_{m,i}x_{m-2,n}\not\in E$ for all $i=1,\ldots,n$. By property~3) we have
$x_{m-2,n-1}^1<x_{m-2,1}^2<x_{m-1,1}^2$. Since there is a factor $x_{m,1}^1x_{m-1,1}^2$ in $w_m$, for all $i=1,\ldots,n$ the inequality $x_{m-2,n-1}^1<x_{m,i}^1$ holds.
Since $w_{m-1}$ ends with a factor $x_{m-2,n-1}^3x_{m-1,n-1}^3$ and $x_{m-1,n-1}^3$ was substituted by $ x_{m,n}^3x_{m-1,n-1}^3x_{m,n-1}^3$ in $w_m$, we have $x_{m-2,n-1}^3>x_{m,i}^3$ for all $i=1,\ldots, n-2$ and  by definition, $x_{m,i}x_{m-2,n-1}\not\in E$ for these $i$. Finally, $w_{m-1}$ and $w_m$ have factors 
$ x_{m-1,n-1}^1x_{m-1,n-2}^2x_{m-2,n-1}^2$ and $x_{m,n}^1x_{m-1,n}^2$, respectively. By property~(2), $x_{m-1,n-2}^1<x_{m-1,n}^2$ and hence by property~(1)
$x_{m-2,n-1}^2<x_{m,n}^1<x_{m,n-1}^1$. So, by Observation~\ref{evident}, $x_{m-2,n-1}$ cannot be adjacent to $x_{m,n}$ and $x_{m,n-1}$. So, no edges connect $W_m$ with $w_{m-2}$.

Since \texttt{Od} contains the permutation $x_1^3x_2^3\ldots x_{n-2}^3x_n^3x_{n-1}^3$, we have $[x_{m,j}^2,x_{m,j}^3]\cap W_{m-1}=\{x_{m-1,j}^3\}$ for all $j=1,\ldots,n-2$. 
Clearly,  $[x_{m,n}^1,x_{m,n}^2]\cap W_{m-1}=\{x_{m-1,n}^2,x_{m-1,n-1}^2\}$, but $x_{m,n}^3<x_{m-1,n-1}^3$, and thus $x_{m,n}x_{m-1,n-1}\not\in E$.
Finally,  $[x_{m,n-1}^2,x_{m,n-1}^3]\cap W_{m-1}=\{x_{m-1,n}^3,x_{m-1,n-1}^3\}$. By property~(1), $x_{m-1,n}^2<x_{m-1,n-1}^2$. Since $w_m$ has a factor
$x_{m,n-1}^1x_{m-1,n-1}^2$, we have $x_{m-1,n}^2<x_{m,n-1}^1$. So,  by Observation~\ref{evident}, $x_{m,n-1}x_{m-1,n}\not\in E$.

It remains to verify that the edges $x_{m,j}x_{m-1,j}$ exist for all $j=1,\ldots,n$. The factors $x_{m,j}^1x_{m-1,j}^2$ imply that 
$x_{m-1,j}^1<x_{m,j}^1<x_{m-1,j}^2<x_{m,j}^2$ for all $j$. The inequalities $x_{m,j}^2<x_{m-1,j}^3<x_{m,j}^3$ can be seen directly from the construction. So, 
 $x_{m,j}x_{m-1,j}\in E$ for all $j=1,\ldots, n$.
\end{proof}

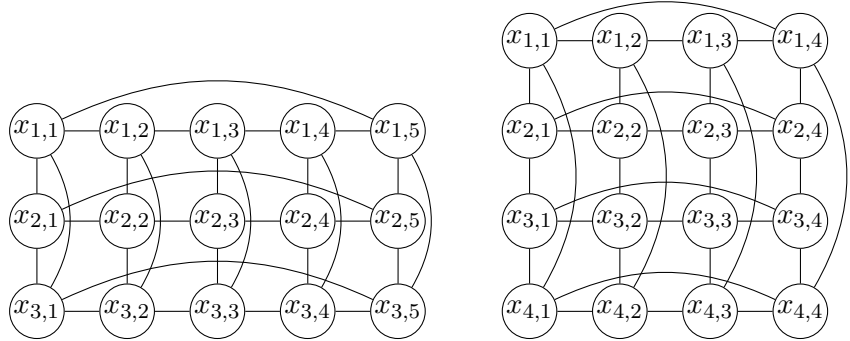
\begin{figure}
\begin{center}
\begin{tabular}{ccc}
\begin{tikzpicture}[scale=1.2, every node/.style={circle, draw, inner sep=0.5pt}]
  \def\rows{3}
  \def\cols{5}

  \foreach \i in {1,2,3} {
    \foreach \j in {1,2,3,4,5} {
      \node (x\i\j) at (\j,-\i) {$x_{\i,\j}$};
    }
  }

  \foreach \i in {1,2,3} {
    \foreach \j in {1,2,3,4} {
      \draw (x\i\j) -- (x\i\the\numexpr\j+1\relax);
    }
  }

  \foreach \i in {1,2} {
    \foreach \j in {1,2,3,4,5} {
      \draw (x\i\j) -- (x\the\numexpr\i+1\relax\j);
    }
  }
  
\draw[bend left=25] (x11) to (x15); 
\draw[bend left=25] (x21) to (x25); 
\draw[bend left=25] (x31) to (x35); 

\draw[bend left=30] (x11) to (x31); 
\draw[bend left=30] (x12) to (x32); 
\draw[bend left=30] (x13) to (x33); 
\draw[bend left=30] (x14) to (x34); 
\draw[bend left=30] (x15) to (x35); 
 
\end{tikzpicture}

& & 

\begin{tikzpicture}[scale=1.2, every node/.style={circle, draw, inner sep=0.5pt}]
  \def\rows{4}
  \def\cols{4}

  \foreach \i in {1,2,3,4} {
    \foreach \j in {1,2,3,4} {
      \node (x\i\j) at (\j,-\i) {$x_{\i,\j}$};
    }
  }

  \foreach \i in {1,2,3,4} {
    \foreach \j in {1,2,3} {
      \draw (x\i\j) -- (x\i\the\numexpr\j+1\relax);
    }
  }

  \foreach \i in {1,2,3} {
    \foreach \j in {1,2,3,4} {
      \draw (x\i\j) -- (x\the\numexpr\i+1\relax\j);
    }
  }

\draw[bend left=25] (x11) to (x14); 
\draw[bend left=25] (x21) to (x24); 
\draw[bend left=25] (x31) to (x34); 
\draw[bend left=25] (x41) to (x44); 

\draw[bend left=30] (x11) to (x41); 
\draw[bend left=30] (x12) to (x42); 
\draw[bend left=30] (x13) to (x43); 
\draw[bend left=30] (x14) to (x44); 


\end{tikzpicture}

\end{tabular}

\caption{The toroidal grid graphs $\mbox{TGr}_{3,5}$ and $\mbox{TGR}_{4,4}$}\label{3x5-torus-grid-graph}
\end{center}
\end{figure}

\section{Open problems}\label{open}

\noindent
A {\em toroidal grid graph} $\mbox{TGr}_{m,n}$ is obtained from the cylindric grid graph $\mbox{CGr}_{m,n}$ by adding edges $x_{1,j}x_{m,j}$ for all $j=1,\ldots,n$, where we assume $m,n\geq 3$. In other words, $\mbox{TGr}_{m,n}$  is the graph formed from the Cartesian product of the cycle graphs $C_m$ and $C_n$. The toroidal grid graphs $\mbox{TGr}_{3,5}$ and  $\mbox{TGr}_{4,4}$ are shown in Figure~\ref{3x5-torus-grid-graph}.  Since $C_n$ is word-representable, $\mbox{TGr}_{m,n}$ is word-representable as the Cartesian product of two word-representable graphs \cite{KL15}. However, what can be said about $\mathcal{R}(\mbox{TGr}_{m,n})$? We were able to find 3-representants for $\mbox{TGr}_{3,3}$ and $\mbox{TGr}_{3,4}$ (for brevity, we rename the vertices $x_{1,1}$, $x_{2,1}$, $x_{3,1}$,  $x_{1,2}$, $x_{2,2}$, $x_{3,2},\ldots$ as $a$, $b$, $c, d, e,f,\ldots$, respectively):
$$w(\mbox{TGr}_{3,3})=abcdefgadhigbcaehbfdeighcfi;$$
$$w(\mbox{TGr}_{3,4})=ajbkcdaeblfcgdjahkigehfdbelcifjgkhli.$$
However, no 3-representation for larger toroidal grid graphs are known. So, we would like to pose the following conjecture.

\begin{conjecture}
If $n,m\ge 3$ and $m+n\ge 8$ then  $\mathcal{R}(\mbox{TGr}_{m,n})\ge 4.$
\end{conjecture}

Note that the toroidal grid graph  $\mbox{TGr}_{4,4}$ is isomorphic to a 4-dimensional cube, for which a 4-representation is known \cite{BroZan19},  but the (non)-existence of 3-representation is an open problem \cite{HHMO24}.

\section*{Acknowledgments} The third author's work was supported by the research project of the Sobolev Institute of Mathematics (project FWNF-2022-0019).

\end{document}